\documentclass[12pt]{article}

\usepackage{geometry}
\usepackage{url}
\usepackage{amsmath,amscd}
\usepackage{amssymb}
\usepackage{amsthm}
\newtheorem{thm}{Theorem}[section]
\newtheorem{lm}[thm]{Lemma}
\newtheorem{prop}[thm]{Proposition}
\newtheorem{cor}[thm]{Corollary}

\theoremstyle{definition}
\newtheorem{define}[thm]{Definition}
\newtheorem{cl}[thm]{Claim}

\newcommand{\cc}{\epsilon}
\newcommand{\FPdim}{\operatorname{FPdim}}
\newcommand{\Vect}{\operatorname{Vec}}
\newcommand{\Rep}{\operatorname{Rep}}
\newcommand{\CoH}{\operatorname{H}}
\newcommand{\End}{\operatorname{End}}
\newcommand{\CoZ}{\operatorname{Z}}
\newcommand{\inv}{\operatorname{Inv}}
\newcommand{\Aut}{\operatorname{Aut}}
\newcommand{\Ker}{\operatorname{Ker}}
\newcommand{\tr}{\operatorname{tr}}
\newcommand{\id}{\operatorname{Id}}
\newcommand{\im}{\operatorname{Im}}
\newcommand{\gal}{\operatorname{Gal}}
\newcommand{\pic}{\operatorname{Pic}}
\newcommand{\fun}{\operatorname{Fun}}
\newcommand{\cC}{\mathcal{C}}
\newcommand{\cD}{\mathcal{D}}
\newcommand{\ot}{\otimes}
\newcommand{\ZZ}{\mathbb{Z}}
\newcommand{\CC}{\mathbb{C}}

\newcommand{\omitted}[1]{}

\theoremstyle{remark}
\newtheorem*{rem}{Remark}

\title{On the classification of certain fusion categories}
\author{David Jordan, Eric Larson}
\date{}

\begin{document}

\maketitle

\begin{abstract}

We advance the classification of fusion categories in two directions.
Firstly, we completely classify integral fusion categories --- and
consequently, semi-simple Hopf algebras --- of dimension $pq^2$, where
$p$ and $q$ are distinct primes.  This case is especially interesting
because it is the simplest class of dimensions where not all integral
fusion categories are group-theoretical.  Secondly, we classify a
certain family of  $\ZZ/3\ZZ$-graded fusion categories, which are
generalizations of the $\ZZ/2\ZZ$-graded Tambara-Yamagami categories.
Our proofs are based on the recently developed theory of extensions of
fusion categories.  
\end{abstract}

\section{Introduction and results}

Recall that a fusion category is a semi-simple rigid monoidal category with
finitely many simple objects  $X_i$, with simple unit $X_0=\mathbf{1}$,
such that $\forall i$, $\End(X_i)=\CC$.  The goal of this
paper is to obtain two classification results for fusion categories, and
to apply one of them to the classification of semi-simple Hopf algebras.
We begin by explaining how the main results of the paper fit into the
existing literature.

Classification of fusion categories is an important and difficult problem.
More specifically, by the Ocneanu rigidity theorem (see \cite{ofc}),
there are finitely many fusion categories of a given dimension and, in
particular, finitely many fusion categories with a given fusion ring $R$
(called \emph{categorifications} of $R$).  This leads to the natural
problems of classifying all fusion categories of a given dimension and
categorifications of a given fusion ring.
In full generality these problems are very hard; for example, the first
problem includes the
classification of finite groups and Lie groups.  However, for certain
dimensions and certain fusion rings these problems are sometimes tractable
and lead to interesting results.  For example, there exist classifications
of fusion categories of dimension $p$, $p^2$ (see \cite{ofc}), $pq$
(\cite{pq}) or $pqr$ (\cite{ENO}), where $p, q, r$ are distinct primes.
In \cite{ty} categorifications of Tambara-Yamagami rings are classified.
In addition, there is a simple description of \emph{group-theoretical}
categories, i.e.\ categories that are Morita equivalent to a pointed
category, (see \cite{ofc}), and all categories whose dimension is a
prime power are of this type (see \cite{DGNO}).

In this paper we extend these classification results in two
directions. Our first result is the classification of integral fusion
categories of dimension $pq^2$, where $p$ and $q$ are distinct primes.
This case is interesting because it is the first class of dimensions
for which an integral fusion category need not be group-theoretical, and
significantly new methods are needed to get a classification.\footnote{%
All integral fusion categories of dimension $pqr$ classified in a recent
preprint \cite{ENO} are group-theoretical and thus the techniques of
\cite{ENO} do not work in our situation.} The second result is the
classification of categorifications of certain fusion rings $R_{p,G}$,
associated with a finite group $G$. These rings are generalizations
of the Tambara-Yamagami rings (which correspond to $p=2$). Namely, we
obtain a complete classification of categorifications of such rings
for $p=3$, when the order of the group $G$ is not divisible by $3$.
Our main results are Theorems \ref{pq2} and \ref{thmR3A} below.

\begin{thm} \label{pq2}
Let $p $ and $q$ be primes, and $\cC$ be an integral\footnote{If $p$ and
$q$ are odd, then the assumption that $\cC$ is integral is redundant.}
fusion category of Frobenius-Perron dimension $pq^2$.  Then, exactly
one of the following is true:
\begin{itemize}
\item $\cC$ is a group-theoretical category.
\item $p=2$, and $\cC$ is a Tambara-Yamagami category \cite{ty}
corresponding to an anisotropic quad\-ratic form\footnote{i.e, of the
form $x^2-ay^2$, where $a$ is a quadratic non-residue.} on $(\ZZ/q\ZZ)^2$;
there are two equivalence classes of such categories.
\item 
\sloppy
The prime $p$ is odd and divides $q + 1$, and $\cC$ is one of
the categories $\cC(p,q,\{\zeta_1,\zeta_2\},\xi)$ we shall explicitly
construct.  Here, $\zeta_1 \neq \zeta_2 \in\mathbb{F}_{q^2}$ are such
that $\zeta_1^p = \zeta_2^p = 1$, but $\zeta_1\zeta_2 \neq 1$, and 
$\xi \in \CoH^3(\ZZ/p\ZZ, \CC^*)\cong \ZZ/p\ZZ$.  There are $(p^2 - p)/2$
equivalence classes of such categories.
\end{itemize}
\end{thm}

\fussy

\begin{cor}\label{Hopf}
All semi-simple Hopf algebras of dimension $pq^2$ are group-theoretical.
\end{cor}

\begin{rem} 
Another proof of Corollary \ref{Hopf}, based on different
methods, is given in \cite{ENO}.
\end{rem}

Our second theorem concerns categorifications of a certain fusion ring
$R_{p,G}$ attached to a finite group $G$ and a prime $p$.

\begin{define} Let $G$ be a finite group whose order is a
square\footnote{This assumption is unnecessary when $p = 1$ or $p = 2$.}
and let $p \in \mathbb{N}$. Then, the \emph{fusion ring} $R_{p, G}$ is
the ring generated by the group ring $\ZZ[G]$ and $X_1,\ldots,X_{p-1}$,
with relations
\begin{align*}
g\otimes X_i=X_i\otimes g=X_i,\quad X_i^* = X_{p-i} \\
X_i \otimes X_j = \begin{cases} 
\sqrt{|G|}X_{i+j} & \mbox{if $i + j \neq p$}, \\ 
\sum_{g \in G}{g} & \mbox{if $i + j = p$}. \end{cases}
\end{align*}
\end{define}

\begin{thm} \label{thmR3A}
Let $A$ be a finite group of order not divisible by $3$.  Then the
fusion ring $R_{3,A}$ admits categorifications if, and only if, $A$ is
abelian, of the form $A\cong \bigoplus_{i=1}^N (\ZZ/p_i^{n_i}\ZZ)^{a_i}$,
where $p_i$ are primes, pairs $(p_i,n_i)$ are distinct, and $a_i$ are
even integers.  In this case, there are 
$3\prod_i{\left(\frac{a_i}{2} + 1\right)}$
categorifications.
\end{thm}
The proofs of these theorems are based on the solvability 
of fusion categories of dimension $p^aq^b$ proved in \cite{ENO}, 
the new theory of extensions of fusion categories developed in 
\cite{ENOnew}, and some intricate linear algebra over finite fields.

The organization of this paper is as follows.  Section \ref{prelims}
contains a review of standard definitions, and also results from recent
literature which we will need.  In Section \ref{pq2sec}, we exhibit
non-trivial gradings on the fusion categories of study, and analyze these
gradings with the methods from \cite{ENOnew}.  Sections \ref{pfthmpq2}
and \ref{Hopfpf} present the proofs of Theorem \ref{pq2} and Corollary
\ref{Hopf}, respectively.  In Section \ref{R3A}, we focus on the case
of $\ZZ/3\ZZ$-graded extensions of $\Vect_A$.  In Section \ref{pfthmR3A}
we prove Theorem \ref{thmR3A}.

\begin{rem}
It is not difficult to extend Theorem \ref{pq2} to classify
non-integral categories of dimension $pq^2$.  This is because the
dimensions of all objects in such a category are necessarily square roots
of positive integers.  This forces a $\ZZ/2\ZZ$-grading on the category,
which means either $p$ or $q$ is $2$. Then a case by case analysis yields
a complete list.  We have not included these computations as they aren't
of particular interest.  
\end{rem}

\begin{rem}
It is also possible to extend Lemmas \ref{VectA} and
\ref{q-gp}, and thus Theorem \ref{thmR3A} to the situation where $A$
contains no elements of order $9$ (i.e.\ the $3$-component of $A$ is a
vector space over $\mathbb{F}_3$).  However, it seems that our methods
break down if $A$ has a more complicated $3$-component.
\end{rem}

\begin{rem}Some possible directions of future research would be a
generalization of Theorem \ref{pq2} to categories of dimension $pq^n$,
$n\geq 3$, and of Theorem \ref{thmR3A} to $p>3$.  These problems reduce
to describing orbits of actions of certain reductive subgroups of
$O(2n,\mathbb{F}_q)$ on the Lagrangian Grassmanian.  While in general
these problems may be intractable, we think that under reasonable
simplifying assumptions one can get manageable and interesting
classifications.\end{rem}

\subsection*{Acknowledgements} The authors would like to warmly thank Pavel
Etingof and Victor Ostrik for posing the problem, and for many helpful
conversations as the work progressed.  We are grateful to Victor Ostrik
for explaining to us how Corollary \ref{Hopf} could be easily derived
from our results.  The work of both authors was supported by the Research
Science Institute, and conducted in the Department of Mathematics at MIT.

\section{Preliminaries}\label{prelims}
In this section, we recall several basic notions about fusion
categories. For more details, see \cite{ofc, ENO, ENOnew, Os1}.  For the
remainder of the paper, $\cC$ and $\cD$ are fusion categories, $G$
is any finite group, and $A$ is a finite abelian group.
 \renewcommand{\labelenumi}{(\roman{enumi})}
\begin{define} A \emph{fusion} (or \emph{based}) ring $R$ is an
associative ring which is free of finite rank as a $\ZZ$-module,
with fixed $\ZZ$-basis $B=\{X_i\}$ containing $X_0=\mathbf{1}$, and an
involution \mbox{$*:B\to B$} extending to an anti-involution $R\to R$,
such that:
\begin{enumerate}
\item for all $i,j,\ X_iX_j = \sum N(i,j,k)X_k,$ 
where $N(i,j,k)$ are non-negative integers, 
\item $N(i,j^*,0)=\delta_{ij}$.
\end{enumerate}
\end{define}

\begin{define} 
The fusion ring of $\cC$, denoted $K(\cC)$, has as its
basis the isomorphism classes of simple objects of $\cC$, with $N(i,j,k)$
equal to the multiplicity of $X_k$ in $X_i\ot X_j$, and $*$ defined by
the duality in $\cC$.  A \emph{categorification} of a fusion ring $R$
is a fusion category $\cC$ such that $K(\cC)=R$.
\end{define}
A fusion ring can have more than one categorification, or none at all. For
example, consider the group ring $\ZZ[G]$ of a finite group $G$ (with
basis $\{g \in G\}$).  Categorifications of these rings are known as
\emph{pointed categories}.  One such categorification is the category of
$G$-graded vector spaces, \mbox{$V=\oplus_{g\in G}V_g$}, with the trivial
associativity isomorphism. We can construct other categorifications by
letting the associativity isomorphism $\alpha$ be defined on the graded components by 
$$\begin{CD}(U_g \otimes V_h) \otimes W_k @>\xi(g,h,k)>> U_g \otimes (V_h \otimes W_k)\end{CD},$$ 
for some 3-cocycle $\xi \in \CoZ^3(G,\CC^*)$.  
We denote the resulting category $\Vect_{G, \xi}$ (or just
$\Vect_{G}$ if $\xi$ is trivial).  It is well known that $\Vect_{G,\xi}$
depends only on the cohomology class of $\xi$, and these categories are
the only pointed categories. Thus,  categorifications of $\ZZ[G]$ are
parameterized by the set $\CoH^3(G,\CC^*)/\Aut(G)$. On the other hand,
the two-dimensional fusion ring with the basis $\{\mathbf{1},X\}$ and
the fusion rules $X^2= \mathbf{1}+nX$ \omitted{ for some $n\ge 0$,} has
two categorifications when $n=0,1$, and no categorifications for $n>1$
(see \cite{1nx}).

\begin{define} 
The Frobenius-Perron dimension $\FPdim X_i$ of $X_i$ is the
largest positive eigen\-value of the matrix $N_i$ with entries $N(i,j,k)$
(such an eigenvalue exists by the Frobenius-Perron theorem).
\end{define}
For categories of group representations (or more generally,
representations of a semi-simple quasi-Hopf algebra), this is the vector
space dimension; however, in general the dimension need not be an integer
--- it is only an algebraic integer.  If all $\FPdim X_i$ are integers,
we call the category \emph{integral}.

\begin{define}
The \emph{dimension} $|\cC|$ of $\cC$ is the sum of the squares of
dimensions of all simple objects of $\cC$.
\end{define}
For the category of representations of a semi-simple quasi-Hopf algebra
$H$, $|\cC|$ is the vector space dimension of $H$, by Maschke's theorem.

\begin{define} We say that
$\cC$ is \emph{graded} by $G$ if $\cC$ decomposes as a direct sum 
$\cC = \bigoplus_{g \in G} \cC_g$, such that 
$\cC_g \otimes \cC_h \subseteq \cC_{gh}$.  
Let us denote the trivial component of the grading $\cC_e$.
When $G$ is abelian, we refer to the trivial component as $\cC_0$.
By an \emph{extension} of $\cD$ by $G$, we mean a $G$-graded fusion
category $\cC$ with $\cC_e=\cD$.
\end{define}

\begin{lm}\label{veca} \cite{mcdd}
Let $\cC$ be a $G$-graded fusion category whose trivial component $\cC_e$
is pointed, with some component $\cC_g$ containing a unique simple object.
Then $\cC_e\cong \Vect_{A}$, with $A$ abelian.
\end{lm}
\begin{proof}
The category $\cC_g$ defines a fiber functor on $\cC_e$, which implies
that $\cC_e=\Rep H$ for some commutative Hopf algebra
$H$. Thus, $\xi=0, H=\fun(G)$.  Now, $\Rep H^*$
is the dual category to $\cC_e$ with respect to $\cC_g$, which is the
same as $\cC_e$. Thus $H^*$ is commutative and $G$ is abelian.
\end{proof}
\begin{define}
The \emph{Picard group} of $\cC$, denoted $\pic (\cC)$, is the set of
all equivalence classes of invertible $\cC$-bimodule categories under
the operation of the tensor product.  Thus, $\pic (\cC)$ is the group
of equivalence classes of Morita autoequivalences of $\cC$.
\end{define}

\begin{define}
Two fusion categories $\cC$ and $\cC'$ are \emph{equivalent}, if there
is an invertible\footnote{by an invertible functor, we mean a functor
with a quasi-inverse.} tensor functor: $\cC \to \cC'$.
If we have two categories $\cC$ and $\cC'$ graded by the same group $G$,
then we say that they are \emph{grading-equivalent} if there is some
invertible tensor functor: $\cC \to \cC'$ which restricts to a functor
$\cC_g \to \cC'_g$ for each $g \in G$.
\end{define}

\begin{thm}\cite{ENOnew}
$\pic (\Vect_{A})$ is the split orthogonal group $O(A \oplus A^*)$.
\end{thm}
For completeness, let us sketch a proof.  The key point is that Morita
equivalences between fusion categories $\cC$ and $\cD$ are in bijection
with braided equivalences between their Drinfeld centers $Z(\cC)$ and
$Z(\cD)$, and this equivalence maps tensor product of bimodule categories
to composition of functors.  (See \cite{ENO}, Theorem 3.1.) In particular,
the group of Morita auto-equivalences of $\cC$ is naturally isomorphic
to the group of braided autoequivalences of $Z(\cC)$.  In the case
$\cC=\Vect_A$, $Z(\cC)$ is $\Vect_{A\oplus A^*}$, with braiding given
by the standard (split) quadratic form.  Thus, the group of braided
autoequivalences of $Z(\cC)$  is isomorphic to $O(A\oplus A^*)$, and
the result follows.

\begin{thm}\cite{ENOnew} 
Fix a fusion category $\cC_e$.  Then categories $\cC$, graded by $G$,
with trivial component $\cC_e$ are classified, up to grading-equivalence,
by triples $(\rho, h, k)$, where \mbox{$\rho:G\to\pic (\cC_e)$} is a
homomorphism, $h \in \CoH^2(G, \inv(Z(\cC_{e})))$, and 
$k \in \CoH^3(G, \CC^*)$.\footnote{Actually the data $h$ and $k$ belong 
to torsors over the groups $\CoH^2(G, \inv(Z(\cC_{e})))$, and 
$\CoH^3(G, \CC^*)$ respectively rather than to the groups themselves.  
This is a technical point which is not going to matter for our considerations.
Here $\inv(\cD)$ denotes the group of invertible objects of $\cD$.}
There are obstructions $\phi_1(\rho)\in \CoH^{3}(G, \inv(Z(\cC_{e})))$,
and $\phi_2(\rho,h) \in \CoH^4(G, \CC^*)$ which must vanish for
$\cC(\rho,h,k)$ to exist.  Here, we consider this data up to the action
of the group of tensor autoequivalences of $\cC_e$.
\end{thm}

\section{\boldmath Cyclic extensions of $\Vect_A$}\label{pq2sec}
Let us fix primes $p$ and $q$, and a generator $\cc$ of $\ZZ/p\ZZ$.
\begin{prop}\label{catRpA} 
Let $|A|$ be coprime to $p$. Then categorifications of $R_{p,A}$ are
parameterized by the data $(\rho,\xi)$, where $\xi \in\CoH^3(\ZZ/p\ZZ,
\CC^*) \simeq \ZZ/p\ZZ$, and $\rho:\ZZ/p\ZZ\to O(A \oplus A^*)$
is a homomorphism such that if we write
$$\rho(i) = \left(\begin{array}{cc}
\alpha_i & \beta_i \\
\gamma_i & \delta_i
\end{array}\right),$$
where $\alpha_i: A \to A, \beta_i: A^* \to A, \gamma_i: A \to A^*, \delta_i: A^* \to A^*$, 
then $\beta_i$ is an isomorphism for all $i \neq 0$. 
Two such categorifications are equivalent if, and only if, they are
related by the natural action of $\Aut(\ZZ/p\ZZ)$ and the subgroup of
the orthogonal group of elements of the form
$$\mathbf{\nu}=\left(\begin{array}{cc}
\psi & 0 \\
\varphi & \psi^{-1*}
\end{array}\right),$$
where $\psi^*\varphi$ is skew-symmetric.
\end{prop}

\begin{proof}
Clearly, any categorification $\cC$ of the fusion ring $R_{p,A}$ is
$\ZZ/p\ZZ$-graded. From Lemma~\ref{veca}, $\cC_0 = \Vect_{A}$.  We must
have that $\beta_i$ is an isomorphism for all $i \neq 0$ since 
$\cC_i = \langle X_i \rangle$, and $\FPdim(X_i)^2 = |A| = |\im \beta_i|$.  Now,
$|\inv(Z(\cC_0))|$ divides $|Z(\cC_0)|$, since the dimension of any
subcategory divides the dimension of the category.  But, $|Z(\cC_0)| = |\cC_0|^2$.  
Because we are assuming that $p$ is coprime to $|A|$, $p$ is
coprime to $|\inv(Z(\cC_0))|$, whereby $\CoH^*(\ZZ/p\ZZ, \inv(Z(\cC_0))) = 0$. 
(In particular this implies that the third cohomology group
is trivial.)  We also have no choice for the second piece of data,
since the second cohomology group is also trivial.  Finally, the second
obstruction vanishes because $\CoH^4(\ZZ/p\ZZ, \CC^*) = 0$.  Therefore,
by \cite{ENOnew}, such categories are determined up to grading-equivalence
by the data $(\rho,\xi)$.

It is clear that if $\cC', \cC''$ are two categorifications of $R_{p,A}$,
then any equivalence between $\cC'$ and $\cC''$, will preserve the
grading, up to the action of $\Aut(\ZZ/p\ZZ)$.  Thus, since the subgroup
of the orthogonal group which acts on the data $(\rho,\xi)$ is the group
of autoequivalences of $\cC_0 = \Vect_{A}$, we conclude the statement
of this proposition.
\end{proof}

Now we consider what happens when instead of requiring that each
graded component $\cC_g$ ($g \neq 0$) contains a unique simple object,
we only require that the graded component $\cC_\cc$ contains a unique
simple object.
Since this condition is not invariant under the action of
$\Aut (\ZZ/p\ZZ)$, we classify these categorifications up to
grading-equivalence.

\begin{thm}\label{deltazero} 
Let $p \in \mathbb{N}$ be relatively prime to $|A|$.  Then,
$\ZZ/p\ZZ$-graded categories $\cC$ with trivial component $\Vect_{A}$
such that $\cC_\cc$ contains a unique simple object, are parameterized up
to grading-equivalence by an element of $\CoH^3(\ZZ/p\ZZ, \CC^*) \simeq \ZZ/p\ZZ$, 
together with a map $\alpha: A \to A$, and an isomorphism
$\gamma: A \to A^*$, such that $\gamma^*\alpha$ is skew-symmetric, and
\begin{equation} \label{qpower}
\left(\begin{array}{cc}
\alpha & \id \\
(\gamma^{-1}\gamma^*)^{-1} & 0
\end{array}\right)^p = \id.
\end{equation}
\noindent
\end{thm}

\begin{proof}
As in Proposition \ref{catRpA}, our categories are determined by the data
$(\rho,\xi)$, except that we only require $\beta_1$ to be an isomorphism.

To specify the homomorphism $\rho$, it suffices to give the image of
the generator $\rho(\cc)$, say:
\begin{equation} \label{M}
\mathbf{M} = \left(
\begin{array}{cc}
\alpha & \beta \\
\gamma & \delta 
\end{array}
\right) \in O(A \oplus A^*)
\end{equation}
such that $\beta$ is invertible and $\mathbf{M}^p = \id$.
However, we must consider such matrices $\mathbf{M}$
up to conjugation by elements of the form:
\begin{equation} \label{conj}
\left(
\begin{array}{cc}
\id & 0 \\
\varphi & \id
\end{array}
\right),
\end{equation}
where $\varphi$ is skew-symmetric, since conjugation by elements of the form 
$$\left(
\begin{array}{cc}
\psi & 0 \\
0 & {\psi^*}^{-1}
\end{array}
\right)$$
amounts to the change of basis 
$(\alpha, \gamma) \mapsto (\psi^{-1}\alpha \psi, \psi^*\gamma \psi)$.

\begin{cl} 
For a given matrix in the form (\ref{M}), there is exactly one matrix
of the form (\ref{conj}) which conjugates it into a matrix where $\delta = 0$.
\end{cl}
\begin{proof}
Observe that when we conjugate (\ref{M}) by (\ref{conj}) we obtain
$$
\left(
\begin{array}{cc}
\id & 0 \\
\varphi & \id
\end{array}
\right) \\
\left(
\begin{array}{cc}
\alpha & \beta \\
\gamma & \delta
\end{array}
\right) \\
\left(
\begin{array}{cc}
\id & 0 \\
-\varphi & \id
\end{array}
\right) \\
= \\
\left(
\begin{array}{cc}
\alpha - \beta\varphi & \beta \\
\star & \varphi \beta + \delta
\end{array}
\right)$$ 
So, a matrix in the form of (\ref{conj}) conjugates $\mathbf{M}$ into a
matrix where $\delta = 0$ if, and only if, 
$\varphi = -\delta\beta^{-1} = -\beta^{-1*}\beta^*\delta\beta^{-1}$, 
which is skew-symmetric because $\mathbf{M}\in O(A\oplus A^*)$.
\end{proof}
Thus, we have reduced the problem to classifying the set of all matrices
$\mathbf{M}$ in the form:
$$\left(\begin{array}{cc}
\alpha & \beta \\
\gamma & 0 
\end{array}
\right) \in O(A \oplus A^*),$$ 
whose $p$\textsuperscript{th} power is the identity matrix.  The condition
that $\mathbf{M} \in O(A \oplus A^*)$ can be expressed as $\gamma^*\alpha$
being skew-symmetric, and $\beta = {\gamma^*}^{-1}$.  Therefore, we
want to find linear maps $\alpha, \gamma$, with $\gamma$ invertible,
such that $\gamma^*\alpha$ is skew-symmetric and
\begin{equation}\left(
\begin{array}{cc}
\alpha & {\gamma^*}^{-1} \\
\gamma & 0
\end{array}
\right) ^ p = \id \label{cond}.\end{equation}
Now, conjugating $\mathbf{M}$ by anything in the general linear group
does not change the property that $\mathbf{M}^p = \id$, so we can replace
$\mathbf{M}$ with
$$\left(\begin{array}{cc}
\id & 0 \\
0 & \gamma^{*-1}
\end{array}\right) \left(
\begin{array}{cc}
\alpha &  {\gamma^*}^{-1} \\
\gamma & 0
\end{array}
\right) \left(\begin{array}{cc}
\id & 0 \\
0 & \gamma^*
\end{array}\right) = \left(\begin{array}{cc}
\alpha & \id \\
(\gamma^{-1}\gamma^*)^{-1} & 0
\end{array}\right)$$
Therefore (\ref{cond}) may be replaced with the condition (\ref{qpower}).
\end{proof}

\begin{rem} 
One may easily deduce Theorem 3.2 of \cite{ty}, 
as a corollary to Theorem \ref{deltazero} above. 
\end{rem}

\begin{lm} \label{uniqueskew}
Let $C$ be a cyclic $q$-group.  Then, up to equivalence, there is
exactly one non-degenerate skew-symmetric map 
$\gamma: C \oplus C \to (C \oplus C)^*$.
\end{lm}

\begin{proof}
Identical to the proof of the corresponding theorem for vector spaces.
\end{proof}

\begin{lm}\label{uniquegamma} Let $a \in \ZZ$, such that $q \not | ~a^2 - 4$,
and let $C$ be a cyclic $q $-group. Then, up to equivalence, there is exactly one
isomorphism $\gamma: C \oplus C \to (C \oplus C)^*$ such that
$(\gamma^{-1}\gamma^*)^2 = a\gamma^{-1}\gamma^* - \id$.
\end{lm}

\begin{proof}
Write $C = \ZZ/q^n\ZZ$.  First, assume to the contrary that
$q^{n-1}\gamma^{-1}\gamma^*$ is multiplication by some constant
$q^{n-1}\lambda$.  Then, 
$q^{n-1}\gamma = (q^{n-1}\gamma^*)^* = (q^{n-1} \lambda \gamma)^* = q^{n-1} \lambda^2 \gamma$.  
Thus, modulo $q$, 
$\lambda^2 = 1$, so $q^{n-1}\id = q^{n-1}(\lambda\id)^2 = q^{n-1}(a\lambda\id - \id) = q^{n-1}(a\lambda - 1)\id$.  
Thus, $1 = \lambda^2 = (2/a)^2$ modulo $q $, which contradicts our
assumption that $q $ does not divide $a^2 - 4$.  Thus, the characteristic
polynomial of $x = \gamma^{-1}\gamma^*$ is $t^2 - at + 1$, since that
is the minimal polynomial and of the correct degree.  Now, we claim that
there is exactly one equivalence class for $x$, namely the class of
\begin{equation} \label{xclass}
x = \left(\begin{array}{cc}
a & 1 \\
-1 & 0
\end{array}\right).
\end{equation}
Since the characteristic polynomial of $x$ is $t^2 - at + 1$, $x$ has the form
$$x = \left(\begin{array}{cc}
a-d & b \\
c & d
\end{array}\right),$$
where $d(a-d) - bc = 1$, and $a \neq 2d$ modulo $q $ if $b = c = 0$ modulo $q$,
because we already proved that $q^{n-1}x$ was not multiplication by a scalar.
If $d(a-d) - bc = 1$, it is straight-forward to check
$$\left(\begin{array}{cc}
a-d & b \\
c & d
\end{array}\right) \left(\begin{array}{cc}
(a-d)y + b & y \\
cy + d & 1
\end{array}\right) = \left(\begin{array}{cc}
(a-d)y + b & y \\
cy + d & 1
\end{array}\right) \left(\begin{array}{cc}
a & 1 \\
-1 & 0
\end{array}\right).$$
Thus, we are done if there is $y \in \mathbb{F}_q $ so that
$$-cy^2 + (a - 2d)y + b = \det \left(\begin{array}{cc}
(a-d)y + b & y \\
cy + d & 1
\end{array}\right) \neq 0,$$
as we can take an arbitrary lift of $y$ into $\ZZ/q^n\ZZ$ to finish.
Since we cannot have $b = c = 0, a - 2d = 0$ modulo $q$, there is some
$y \in \mathbb{F}_q$ that finishes the claim, unless 
$q = 2$ and $-cy^2 + (a - 2d)y + b = y^2 + y$.  
But in the latter case, it follows that $\det x = 0$, which contradicts
the invertibility of $x$.  Therefore, we may assume (\ref{xclass}).
If we write
$$\gamma = \left(\begin{array}{cc}
b & c \\
d & e
\end{array}\right),$$
then we have that
$$\left(\begin{array}{cc}
b & d \\
c & e
\end{array}\right) = \gamma^* = \gamma x = \left(\begin{array}{cc}
b & c \\
d & e
\end{array}\right) \left(\begin{array}{cc}
a & 1 \\
-1 & 0
\end{array}\right) = \left(\begin{array}{cc}
ab-c & b \\
ad-e & d
\end{array}\right).$$
Therefore, 
$$\gamma = \left(\begin{array}{cc}
d & (a-1)d \\
d & d
\end{array}\right).$$
But, for any $y, t$, such a map is equivalent to
$$\left(\begin{array}{cc}
ay+t & -y \\
y & t
\end{array}\right)\gamma \left(\begin{array}{cc}
ay+t & y \\
-y & t
\end{array}\right) = \left(\begin{array}{cc}
d(y^2+ayt+t^2) & (a-1)d(y^2+ayt+t^2) \\
d(y^2+ayt+t^2) & d(y^2+ayt+t^2)
\end{array}\right).$$
In order to show the statement of this lemma, it suffices
to show that there are $y, t \in \ZZ/q^n\ZZ$
so that $y^2 + ayt + t^2 = 1/d$. 
By Hensel's Lemma, it suffices to prove there is a solution modulo $q$
such that $2y + at \neq 0$.

If $q = 2$, this is clear, since we can take $y = t = 1$.  ($d$ and
$a$ must both be $1$, since $\gamma$ is invertible.)  If $q \neq 2$,
then it is equivalent to $(2y+at)^2 + (4 - a^2)t^2 = 4/d$. Observe
that we may choose $2y + at$ and $t$ independently; that is we want
to find $z, w \in \mathbb{F}_q $ so that $z^2 + (4 - a^2)w^2 = 4/d$.
If $4/d$ is a quadratic residue, we may choose $z^2 = 4/d, w^2 = 0$.
Therefore, suppose that $d$ is not a quadratic residue. Then, if 
$(4 - a^2)$ is also not a quadratic residue, we may take 
$z = 0, w^2 = 4/((4 - a^2)d)$.  So, say that $4 - a^2$ is a quadratic residue, 
$4 - a^2 = f^2$.  Then, our equation becomes $z^2 + (fw)^2 = 4/d$, where 
$f \neq 0$.  Define the sets $S_i = \{2/d + i, 2/d - i\}$ for 
$i = 0, 1, 2, \ldots \frac{q -1}{2}$.  
By the pigeonhole principle, we must either
have that $2/d$ is a quadratic residue, in which case we are done, or
that two quadratic residues in the same $S_i$, say $\{z^2, (fw)^2\} = S_i$, 
which implies that $z^2 + (fw)^2 = 4/d$.
\end{proof}

\begin{lm} \label{Zpgrading}
Let $\cC$ be an integral fusion category of Frobenius-Perron dimension $pq^2$.
Then, either $\cC$ is faithfully graded by $\ZZ/p\ZZ$, or $\cC$ is group-theoretical.
\end{lm}

\begin{proof}
By \cite{ENO}, any fusion category of dimension $p^mq^n$ is Morita
equivalent to a nilpotent fusion category.  Therefore, every fusion
category of dimension $pq^2$ is either Morita equivalent to a category
with a faithful $\ZZ/q \ZZ$ grading, or one with a faithful $\ZZ/p\ZZ$
grading.

Suppose that $\cC$ is Morita equivalent to a category $\cD$ with a
faithful $\ZZ/q \ZZ$ grading.  Let $\cD_0$ be the trivial component of
the grading. Then, $\cD_0$ is an integral fusion category of dimension
$pq$. Therefore, by \cite{pq}, $\cD_0$ is group-theoretical, and thus
Morita equivalent to a pointed category $\cD_0'$, and $\cD$ is Morita
equivalent to some $\ZZ/q \ZZ$ graded category $\cD'$ whose trivial
component is $\cD_0'$ (see \cite{ENO}, Lemma 3.3).  But the possible
dimensions of objects of $\cD'$ are only $1, \sqrt{q }, \sqrt{p},
\sqrt{q p}$.  Thus, since Morita equivalence preserves integrality,
$\cD'$ is pointed, and therefore $\cC$ is group-theoretical.

Next, suppose that $\cC$ is not Morita equivalent to a category $\cD$ with
a faithful $\ZZ/q \ZZ$ grading, and that it does not possess a faithful
$\ZZ/p\ZZ$ grading. Then, since all fusion categories of dimension $pq^2$
are solvable \cite{ENO}, $\cC$ is an equivariantization of some category
$\cC_0$ of dimension $q^2$ by $\ZZ/p\ZZ$.  Since all integral categories
of dimension $q^2$ are pointed, and any equivariantization of a pointed
category is group-theoretical, the statement of this lemma follows.
\end{proof}

\begin{lm} \label{q2grouptheoretical}
Any integral fusion category of Frobenius-Perron dimension $pq^2$,
which is $\ZZ/p\ZZ$-graded, such that the trivial component of the grading
is $\Vect_{\ZZ/q^2\ZZ, \xi}$, is group-theoretical.
\end{lm}

\begin{proof}
Since our category is integral, it is either pointed, in which case we are
done, or there is an object of dimension $q $. If that is the case, then
by Lemma~\ref{veca}, $\xi = 0$.  Thus, the Picard group of the trivial
component is $O(\ZZ/q^2\ZZ + (\ZZ/q^2\ZZ)^*)$.  Therefore, the category
must be group-theoretical, since $q (\ZZ/q^2\ZZ + (\ZZ/q^2\ZZ)^*)$ is
an invariant Lagrangian subspace, under any action, with respect to the
split quadratic form.
\end{proof}

\begin{lm} \label{commutes}
Let $\alpha$ and $\gamma$ be $2 \times 2$ matrices, with $\gamma$
invertible, and $\alpha^*\gamma$ skew-symmetric.  Then, $\alpha$ commutes
with $\gamma^{-1}\gamma^*$.
\end{lm}
\begin{proof}
Let $\varphi=\alpha\gamma^{-1}$.  Since $\alpha^*\gamma$ is
skew-symmetric, $\varphi$ is as well.  Explicit computation reveals
that $\gamma\varphi \gamma^* = \gamma^* \varphi \gamma$, which implies
$\alpha\gamma^{-1}\gamma^* = \gamma^{-1}\gamma^*\alpha$.
\end{proof}

\section{Proof of Theorem \ref{pq2}}\label{pfthmpq2}
Write $A = (\ZZ/q \ZZ)^2$.  By Lemmas \ref{Zpgrading} and
\ref{q2grouptheoretical}, either $\cC$ is group-theoretical, or
$\cC$ is a $\ZZ/p\ZZ$-graded category with trivial component
$\Vect_{A, \xi}$.  Since $\cC$ has all objects of integral dimension
by assumption, either $\cC$ is pointed, in which case we are done,
or $\cC$ has an object of dimension $q $, in which case by
Lemma~\ref{veca}, $\xi = 0$.  Therefore, if $p = 2$, by \cite{ty}
such categories are parameterized by a quadratic form $\gamma$, and
by \cite{gnn} such categories are group-theoretical if, and only
if, there is a subgroup $L \subset A$, such that $|L| = \sqrt{|A|} = q$, 
such that $\gamma$ is $0$ when restricted to $L$. Since $A$
is a two-dimensional vector space, this is equivalent to the form
$\gamma$ being isotropic. This completes the proof if $p = 2$.
Thus, we will assume that $p$ is odd.

In particular, $\cC$ is a $\ZZ/p\ZZ$-graded category with trivial
component $\Vect_{A}$ such that $\cC_\cc$ contains a single simple object,
where $\cc$ is the generator of $\ZZ/p\ZZ$.  From Theorem \ref{deltazero},
such $\cC$ are parameterized by an element of $\CoH^3(\ZZ/p\ZZ, \CC^*)$,
together with an equivalence class of a pair of maps 
$\alpha: A \to A, \gamma: A \to A^*$, 
where $\gamma$ is an isomorphism, which satisfy
$\gamma^*\alpha$ skew-symmetric, and (\ref{qpower}).

Write $x$ for $\gamma^{-1}\gamma^*$.  From Lemma~\ref{commutes}, we have
that $\alpha$ and $x$ commute. Consider the matrix
$$\mathbf{M} = \left(\begin{array}{cc}
\alpha & \id \\
x^{-1} & 0
\end{array}\right)$$
as a two-by-two matrix over the commutative subring of matrices generated
by $\alpha$ and $x$.  Then, since $\mathbf{M}^p = \id$, we must have
$\det(\mathbf{M})^p = \id$. Since $\det(\mathbf{M}) = -x^{-1}$, we have
that $x^p = -\id$.
Over the algebraic closure of $\mathbb{F}_q $, we may choose a basis such that 
$$x =\left(\begin{array}{cc}
-\mu & 0 \\
0 & -\lambda
\end{array}\right),$$
where $\lambda^p = \mu^p = 1$.  
Since $\det x = \det \gamma^{-1} \det \gamma^* = 1$, 
$\mu = \lambda^{-1}$.  By Lemma~\ref{uniquegamma} there
is exactly one solution, up to equivalence, to the equation $\gamma^*
= \gamma x$, where $\gamma$ is invertible.  This equation is a system
of four linear equations in the entries of the matrix for $\gamma$.
Solving it yields
$$\gamma = \left(\begin{array}{cc}
0 & -\lambda^{-1} g \\
g & 0
\end{array}\right).$$
Write
$$\alpha = \left(\begin{array}{cc}
a & b \\
c & d
\end{array}\right).$$
Then, we have that
$$\gamma^*\alpha = \left(\begin{array}{cc}
0 & g \\
-\lambda g & 0
\end{array}\right) \left(\begin{array}{cc}
a & b \\
c & d
\end{array}\right) = \left(\begin{array}{cc}
cg & dg \\
-a\lambda^{-1} g & -b\lambda^{-1} g
\end{array}\right)$$
is skew-symmetric. In other words,
$$\alpha = \left(\begin{array}{cc}
a & 0 \\
0 & a\lambda^{-1}
\end{array}\right).$$
We must therefore have:
$$\left(\begin{array}{cccc}
a & 0 & 1 & 0 \\
0 & a\lambda^{-1} & 0 & 1 \\
-\lambda & 0 & 0 & 0 \\
0 & -\lambda^{-1} & 0 & 0
\end{array}\right)^p = \id,$$
or equivalently, there exist $\zeta_1, \zeta_2$ distinct
$p$\textsuperscript{th} roots of unity such that
$$\left(\begin{array}{cc}
a & 1 \\
-\lambda & 0
\end{array}\right)$$ 
is conjugate to 
$$\left(\begin{array}{cc}
\zeta_1 & 0 \\
0 & \zeta_2
\end{array}\right),$$
i.e.\ $a = \zeta_1 + \zeta_2, \lambda = \zeta_1\zeta_2$.  Since $p$
is odd, $x$ and $\alpha$ have the same block form, and therefore the
same centralizer.

\begin{cl} 
There exists a basis so that both $x$ and $\alpha$ are matrices over
$\mathbb{F}_q$ if, and only if, $p | q^2 - 1$.
\end{cl}
\begin{proof}
Observe that 
$p | q^2 - 1 \Leftrightarrow |\gal(\mathbb{F}_q [\zeta_p]:\mathbb{F}_q )| \leq 2 
\Leftrightarrow \zeta + \zeta^{-1} \in \mathbb{F}_q$
for each $\zeta$ which is a primitive $p$\textsuperscript{th} root
of unity.  When $\lambda = 1$, $x = -\id$, and 
$\alpha = (\zeta_1 + \zeta_1^{-1})\id$ are central, and are matrices over
$\mathbb{F}_q$ $\Leftrightarrow$ $\zeta_1 + \zeta_1^{-1} \in \mathbb{F}_q$
$\Leftrightarrow$ $p | q^2 - 1$.  When $\lambda \neq 1$, to see the
``only if'', observe that since 
$\lambda + \lambda^{-1} = -\tr x \in \mathbb{F}_q $, 
we know that $p | q^2 - 1$.  To see the ``if'', let
$$\psi=\left(\begin{array}{cc}
\lambda & -1 \\ 
-1 & \lambda
\end{array}\right),$$ 
and observe that
\begin{align}\label{xandalpha}
\psi (\gamma^{-1}\gamma^*) \psi^{-1} &= \left(\begin{array}{cc}
-(\lambda + \lambda^{-1}) & -1\\
1 & 0
\end{array}\right),\\
\psi \alpha \psi^{-1} &= \left(\begin{array}{cc}
(\lambda + \lambda^{-1} + 1)\frac{a}{\lambda+1} & \frac{a}{\lambda + 1} \\
-\frac{a}{\lambda + 1} & \frac{a}{\lambda + 1}
\end{array}\right),\nonumber
\end{align}
both of which are matrices over $\mathbb{F}_q $.
\end{proof}

It is thus clear that if $q^2 - 1$ is not divisible by $p$, there are no
non-pointed categorifications, and if $q^2 - 1$ is divisible by $p$, then
up to grading-equivalence, non-pointed categorifications are determined
by an element of $\CoH^3(\ZZ/p\ZZ, \CC^*)$ together with unordered
pairs $\{\zeta_1, \zeta_2\}$ of distinct $p$\textsuperscript{th}
roots of unity under the equivalence 
$\{\zeta_1, \zeta_2\} \sim \{\zeta_1^{-1}, \zeta_2^{-1}\}$.  
Such pairs determine the pair $(\gamma, \alpha)$ uniquely because they
determine $\gamma$ uniquely, and $\gamma^{-1}\gamma^*$ and $\alpha$
have the same centralizer.

\begin{cl} The category $\cC$ is group-theoretical if, and only if,
$\zeta_1\zeta_2 \in \mathbb{F}_q $.
\end{cl}

\begin{proof}
The resulting category is group-theoretical if, and only if, there exists
a Lagrangian subspace $L \subset A \oplus A^*$, with respect to the split
quadratic form $q(a \oplus b) = ba$, which is invariant under the action
of $\ZZ/p\ZZ$.

Fix the homomorphism $\rho$.  Write $\mathbf{M}=\rho(\cc)$, and let
$\alpha$ and $\gamma$ be as in (\ref{qpower}).  Denote the chosen basis
of $A$ by $e_1, e_2$.  This gives a basis $e_1, e_2, e_1^*, e_2^*$ for
$A \oplus A^*$.  Because $\gamma^{-1}\gamma^*$ and $\alpha$ have the
same centralizer, Lemma \ref{uniquegamma} and equation \ref{xandalpha}
imply that we may assume
$$\gamma = \left(\begin{array}{cc}
-1 & -(\lambda + \lambda^{-1} + 1) \\
1 & -1 
\end{array}\right).$$  
Thus,
\begin{equation}\label{Meqn}
\mathbf{M} = \left(\begin{array}{cccc}
\frac{a(\lambda^2 + \lambda + 1)}{\lambda^{2} + \lambda} & \frac{a}{\lambda + 1} & -\frac{\lambda}{(\lambda+1)^2} &  -\frac{\lambda}{(\lambda+1)^2} \\
-\frac{a}{\lambda + 1} & \frac{a}{\lambda + 1} & \frac{\lambda^2+\lambda+1}{(\lambda + 1)^2} & -\frac{\lambda}{(\lambda+1)^2} \\
-1 & -\frac{(\lambda^2+\lambda+1)}{\lambda} & 0 & 0 \\
1 & -1 & 0 & 0
\end{array}\right).\end{equation}
For any element or subspace $a$ of $A \oplus A^*$, denote by $\pi a$
the projection of $a$ onto $A$.  Fix some Lagrangian subspace $L$.
We consider three cases.

\paragraph{\boldmath Case 1: $\pi L$ has dimension $0$.}
It follows that $L = A^*$, and by inspection,
such Lagrangian subspaces are never invariant subspaces
of the action of $\ZZ/p\ZZ$ by the homomorphism $\rho$.

\paragraph{\boldmath Case 2: $\pi L$ has dimension $1$.}
In this case, we prove that there is an invariant Lagrangian subspace
if, and only if, $\lambda \in \mathbb{F}_q $.

First we will show the ``only if''.  Since $\pi L$ has dimension $1$,
there is some vector $v \neq 0 \in \pi L$.  We claim that $v \in L$. To see
this, let $v'$ be a lift of $v$ to $L$. Write $v' = v + w$.  Since 
$v + w \in L$, it suffices to show that $w \in L$.  Take $w' \in L$, such
that $w' \notin \langle v' \rangle$.  Since $\pi L$ has dimension $1$,
we have that $\pi w' = \lambda v$.  Consider the vector $w' - \lambda v'$.  
We have $\pi(w' - \lambda v') = 0$, so $w' - \lambda v' \in A^*$.
Since $w' - \lambda v' \in L$, we have that
\begin{align*}
0 &= q((w' - \lambda v') + v') \\
&= q((w' - \lambda v') + v + w) \\
&= ((w' - \lambda v') + w)(v) \\
&= (w' - \lambda v')(v), \\
\end{align*}
since $w(v) = q(v') = 0$. But because $w(v) = 0$,
$w \in \langle w' - \lambda v' \rangle \subset L$, since the subspace of $A^*$
which evaluates to $0$ on a non-zero vector in $A$ is one-dimensional.
Therefore, $v \in L$.
It follows that $\mathbf{M}v \in L$, and therefore that
$\alpha v = \pi(\mathbf{M}v) \in \pi L = \langle v \rangle$.
Thus, $v$ is an eigenvector of $\alpha$.  It follows that an eigenvalue of
$\alpha$ lies in $\mathbb{F}_q $.  Since the eigenvalues of $\alpha$ are
$-\lambda$ and $-\lambda^{-1}$, we have that $\lambda \in \mathbb{F}_q$.

In order to see the ``if'', consider $L = \langle v, w \rangle$, where
\begin{align*}
v &= (1, -\lambda, 0, 0), \\
w &= (0, 0, \lambda, 1).
\end{align*}
Since $v \in A, w \in A^*$, and $w(v) = 0$, it is clear that $L$ is a
Lagrangian subspace.  Thus, it suffices to show that $L$ is invariant
under the action of $\ZZ/p\ZZ$, or that 
$\mathbf{M}v, \mathbf{M}w \in \langle v, w \rangle$.  By (\ref{Meqn}),
\begin{align*}
\mathbf{M}v = \frac{a}{\lambda} v + (\lambda + 1)w, \textrm{ and}\quad
\mathbf{M}w = -\frac{1}{\lambda^{-1} + 1}v.
\end{align*}

\paragraph{\boldmath Case 3: $\pi L$ has dimension $2$.}
In this case, we prove that there is no invariant Lagrangian subspace if
$\lambda \notin \mathbb{F}_q$.  Assume to the contrary.  Since $\pi L$
has dimension $2$, $e_1, e_2 \in \pi L$.  Let $v$ be an arbitrary lift
of $e_1$ to $L$, and $w$ be an arbitrary lift of $e_2$ to $L$.  Clearly,
we have $L = \langle v, w \rangle$.  Since $q(v) = 0$, and $\pi v = e_1$,
$v$ must have the form $(1, 0, 0, s)$ for some $s \in \mathbb{F}_q$.
Similarly, $w$ must have the form $(0, 1, s', 0)$.  Since $q(v+w) = 0$,
we have that $s' = -s$.  In other words, our Lagrangian subspace would
have to be the span of two vectors in the form 
$v = (1, 0,  0, s)$, 
$w = (0, 1, -s, 0)$.  
We have $\mathbf{M}v \in \langle v, w \rangle$. Now,
we explicitly compute $\mathbf{M}v=(\kappa,\tau,-1,1)$, where
\begin{align*}
\kappa &= \frac{a(\lambda+1)(\lambda^2+\lambda+1)- s \lambda^2}{\lambda(\lambda+1)^2}, \\
\tau &= -\frac{s \lambda + a(\lambda+1)}{(\lambda+1)^2}.
\end{align*}
Since $\mathbf{M}v \in \langle v, w \rangle$, we have that there exists
$c_v, c_w \in \mathbb{F}_q $ such that $c_vv + c_ww - \mathbf{M}v = 0$.  Since 
$\pi(c_vv + c_ww) = \left(\begin{smallmatrix}  
c_v \\ 
c_w 
\end{smallmatrix}\right)$, 
we know $c_v=\kappa, c_w=\tau$.
Thus, $0=\kappa v + \tau w - \mathbf{M}v=(0,0,1-\tau s, \kappa s -1)$. As
$$1- \tau s =  \frac{(s \zeta_{2} + \zeta_{1} \zeta_{2} + 1)(s \zeta_{1} + \zeta_{1} \zeta_{2} + 1)}{\zeta_1^2 \zeta_2^2 + 2 \zeta_1 \zeta_2 + 1},$$ 
we deduce
$$(s \zeta_2 + \zeta_1 \zeta_2 + 1) (s \zeta_1 + \zeta_1 \zeta_2 + 1) = 0.$$
Without loss of generality we may assume $s = -\zeta_1 -\zeta_2^{-1}$. Then, 
$$0=\kappa s - 1= -\frac{(\zeta_1 + \zeta_2)(\zeta_1 \zeta_2 + 1)^2}{\zeta_1 \zeta_2^2}.$$
But, this is impossible, as $\zeta_1, \zeta_2$ are $p$\textsuperscript{th}
roots of unity, $p$ is odd, $\zeta_1 \zeta_2 \notin \mathbb{F}_q$, and
$\zeta_1 \neq \zeta_2$.  
(The last two assumptions are needed only when $q = 2$.)
\end{proof}

At this point, we can count the number of non-group-theoretical
categories of dimension $pq^2$ up to grading-equivalence, and up
to general equivalence.  If $p$ does not divide $q^2 - 1$, then
all categorifications are pointed. If $p$ divides $q -1$, then all
$p$\textsuperscript{th} roots of unity lie in $\mathbb{F}_q$.  Therefore,
non-group-theoretical categorifications occur only when $p$ divides $q  + 1$; 
we have one such categorification for each pair $\{\zeta_1,\zeta_2\}$
such that $\zeta_1\zeta_2 \neq 1$.

To account for grading-equivalences, we first compute the number
$\frac{(p-1)(p-3)}{4}$ of pairs $\{\zeta_1,\zeta_2\}$ with $\zeta_i\neq 1$
up to equivalence $\{\zeta_1,\zeta_2\}\sim\{\zeta_1^{-1},\zeta_2^{-1}\}$.
To this we add the number $\frac{p-1}{2}$ of pairs $\{1,\zeta\}$ up to
equivalence $\{1,\zeta\}\sim\{1,\zeta^{-1}\}$, for a total of 
$\frac{(p - 1)^2}{4}$ 
non-group-theoretical categories up to grading-equivalence.

To account for general equivalences, consider the action of 
$\Aut \ZZ/p\ZZ$ on our categories. An element $g \in (\Aut \ZZ/p\ZZ)$
acts by multiplication by $g^{-2}$ on $H^3(\ZZ/p \ZZ,\CC^*)$, and
sends $(\zeta_1, \zeta_2) \to (\zeta_1^g, \zeta_2^g)$.  There are
three orbits on $H^3(\ZZ/p \ZZ,\CC^*)$ under this action: the quadratic
non-residues, the quadratic residues, and $0$ in an orbit by itself.
The stabilizer of the first two orbits is $\pm 1$, which then acts on the
pairs $\{\zeta_1,\zeta_2\}$ as in the graded case, giving 
$\frac{(p - 1)^2}{4}$ 
categorifications each, for $\frac{(p - 1)^2}{2}$ together.  The $\{0\}$
orbit in $H^3(\ZZ/p\ZZ,\CC^*)$ yields two types of orbits on the set of
pairs $\{\zeta_1,\zeta_2\}$.  Clearly $\Aut \ZZ/p{Z}$ acts transitively
on pairs $\{1, \zeta\}$.  So suppose that $\zeta_2 = \zeta_1^k$. Then
the set $\{k, k^{-1}\}$ determines the orbit, and the number of such
sets which do not contain $0, \pm 1$ is $\frac{p-3}{2}$.  This gives a
total of $\frac{p^2 - p}{2}$ categorifications.

\section{Proof of Corollary \ref{Hopf}}\label{Hopfpf}
By Lemma \ref{Zpgrading}, all categories of dimension $pq^2$ without
a faithful $\ZZ/p\ZZ$-grading are group-theoretical.  Let us suppose
that $\cC$ of dimension $pq^2$ is faithfully $\ZZ/p\ZZ$-graded, and is
the category of representations of some semi-simple Hopf algebra $H$,
and demonstrate that $\cC$ is group-theoretical.

The faithful $\ZZ/p\ZZ$ grading on $\cC$ induces a faithful
$\ZZ/p\ZZ$-grading on $H^*$ as follows.  Since $\cC$ is faithfully
$\ZZ/p\ZZ$-graded, there exists a central group-like element $c\in H$,
such that $c^p=1$, defining the grading.  This element defines the
decomposition $H^*=\bigoplus_kH^*_k$, where 
\mbox{$H^*_k = \{f \in H^*\textrm{ s.t. } f(cx)=\zeta^kf(x)\}$} 
and $\zeta = e^{\frac{2\pi i}{p}}$. Clearly $H^*_k\neq 0$ for all $k$.

We consider the sub-algebra $H^*_0$ of $H^*$, and we let $\cD$ denote
the category of $H^*_0$-bimodules in $\cC$; $\cD$ is Morita equivalent
to $\cC$.  Because $H^*_0\in \cC_0$, $\cD$ is also $\ZZ/p\ZZ$-graded,
and we have $|\cD_0|=q^2$, so $\cD_0$ is pointed.  Furthermore, $H^*$
is an algebra in $\cD$, whose $0$-component is $H^*_0$, the unit in $\cD$.

\begin{cl} 
The multiplication map 
$\displaystyle{\mu:H^*_k\ot_{H^*_0}H^*_l \to H^*_{k+l}}$ 
is an isomorphism for all $k$ and $l$.
\end{cl} 
\begin{proof}
We claim first that the map
$$\Delta_l=(\pi_l\ot \id)\circ \Delta: H\to H/(c-\zeta^l)\ot H$$
is injective.  Indeed, suppose $a\in H$ s.t. $\Delta_l(a)=0$.  Then for
all $V\in \cC_l,U \in \cC,$ we have that $a|_{V\ot U}=0$.  Taking $U=H$,
we have $V\ot H \cong (\dim V) H$, so $a$ must be zero.  By duality,
\begin{equation}\label{mueqn}
\mu: H^*\ot_{H^*_0} H^*_l\to H^*
\end{equation} 
is surjective.  By the Nichols-Zoeller theorem \cite{NZ}, $H^*$ is free
over $H^*_0$, of rank $p$.  Therefore, the left hand side of (\ref{mueqn})
has dimension $p\cdot\dim H^*_l=\dim H^*$. Thus (\ref{mueqn}) is an
isomorphism.  Restricting to the graded components yields the claim.
\end{proof}
The claim implies that each $H^*_k$ is an invertible object in $\cD$,
and so in particular there are invertible objects in each $\cD_k$, in
addition to the $q^2$ invertible objects in $\cD_0$.  As the number of
invertible objects must divide the overall dimension $pq^2$ of $\cD$,
we conclude that $\cD$ is pointed.

\section{\boldmath Categorifications of $R_{3, A}$}\label{R3A}

\begin{lm} \label{VectA} 
If $|A|$ is coprime to $3$, $\ZZ/3\ZZ$-graded categories $\cC$ with
trivial component $\Vect_{A}$ such that $\cC_\cc$ contains a single
simple object, are all categorifications of $R_{3,A}$, and are, up
to grading-equivalence, parameterized by pairs $(\xi, \gamma)$, where
$\xi$ is an element of $\CoH^3(\ZZ/3\ZZ, \CC^*)$ $\simeq \ZZ/3\ZZ$, and
$\gamma$ is a map $A \to A^*$ such that $\gamma^*\gamma^{-1}\gamma^*$ is
skew-symmetric.  If our equivalence is not required to preserve grading,
we must additionally identify $\gamma$ with $\gamma^*$.
\end{lm}

\begin{proof}
Clearly, categorifications of $R_{3,A}$ are $\ZZ/3\ZZ$-graded categories
$\cC$ with trivial component $\Vect_{A}$ such that $\cC_\cc$ contains
a single simple object.  To see the reverse inclusion, recall that
$\cC_g$ is contains a unique simple object if, and only if, $\rho(g)$
has its upper right entry an isomorphism. Thus, it suffices to show
that $\rho(g)$ has its upper right entry an isomorphism if, and only if,
$\rho(g^{-1})$ does.  But this is clear because $\rho$ is a homomorphism
into the split orthogonal group, so the upper right entry of $\rho(g)$
is the dual of the upper right entry of  $\rho(g^{-1})$.  Therefore, by
Theorem \ref{deltazero}, $\cC$ is determined up to grading-equivalence by
an element of $\CoH^3(\ZZ/3\ZZ, \CC^*) \simeq \ZZ/3\ZZ$, together with a
map $\alpha: A \to A$, and an isomorphism $\gamma: A \to A^*$, satisfying
the relations $\gamma^*\alpha$ is skew-symmetric and (\ref{qpower}).
Write $x$ for $\gamma^{-1}\gamma^*$. To solve equation (\ref{qpower}),
we explicitly compute:
$$\left(\begin{array}{cc}
\id & 0 \\
0 & \id
\end{array}\right) = \left(\begin{array}{cc}
\alpha & \id \\
x^{-1} & 0
\end{array}\right)^3 = \left(\begin{array}{cc}
\alpha^3 + x^{-1}\alpha + \alpha x^{-1} & \alpha^2 + x^{-1} \\
x^{-1}\alpha^2  + x^{-2} & x^{-1} \alpha
\end{array}\right).
$$
In particular, $\alpha = x$. The condition that $\gamma^*\alpha$
is skew-symmetric then becomes that $\gamma^*\gamma^{-1}\gamma^*$ is
skew-symmetric, from which it follows that 
$(\gamma^{-1}\gamma^*)^3 = -\id$.  
When $\alpha = x$, and $x^3 = -\id$, it is not hard to check
that (\ref{qpower}) is satisfied.

In other words, the conditions on $(\alpha, \gamma)$ given by Theorem
\ref{deltazero} are equivalent to $\alpha = \gamma^{-1}\gamma^*$, and
$\gamma^*\gamma^{-1}\gamma^*$ skew-symmetric.  Therefore, the choice of
maps $\alpha$ and $\gamma$ is equivalent to the choice of a single map
$\gamma$ such that $\gamma^*\gamma^{-1}\gamma^*$ is skew-symmetric.

Finally, in the case where we do not require that our equivalence
preserves grading, we must figure out what happens under the action of
$\Aut \ZZ/3\ZZ$.  In order to do this, we must consider what happens to
$\gamma$ under the transformation 
$\mathbf{\mathbf{M}} \to \mathbf{M}^{-1} = \mathbf{M}^*$. In our case,
$$\mathbf{M} = \left(
\begin{array}{cc}
\gamma^{-1}\gamma^* & {\gamma^*}^{-1} \\
\gamma & 0
\end{array}
\right), \mbox{\quad}
\mathbf{M}^* = \left(\begin{array}{cc}
0 & \gamma^{-1} \\
\gamma^* & \gamma\gamma^{-1*}
\end{array}\right).$$
We find
\begin{align*}
\mathbf{M}^* &\sim \left(\begin{array}{cc}
1 & 0 \\
\gamma^*\gamma^{-1}\gamma^* & 1
\end{array}\right) \left(\begin{array}{cc}
0 & \gamma^{-1} \\
\gamma^* & \gamma\gamma^{-1*}
\end{array}\right) \left(\begin{array}{cc}
1 & 0 \\
-\gamma^*\gamma^{-1}\gamma^* & 1
\end{array}\right) \\ &= \left(\begin{array}{cc}
\gamma^{-1*}\gamma & \gamma^{-1} \\
\gamma^* & 0
\end{array}\right),
\end{align*}
which is the same as the matrix $\mathbf{M}$ with $\gamma$ replaced with $\gamma^*$.
\end{proof}

\begin{lm}\label{q-gp} 
Let $q \neq 3$, $A$ be an abelian $q$-group, and $\gamma$ a
non-degenerate map $A \to A^*$ such that $\gamma^*\gamma^{-1}\gamma^*$
is skew-symmetric. Then $A$ may be decomposed as 
$\bigoplus_i{(C_i \oplus C_i)}$, for cyclic groups $C_i$, 
where the $C_i \oplus C_i$ are mutually orthogonal 
with respect to $\gamma$, and on each component $C_i \oplus C_i$, 
either $\gamma$ is skew-symmetric, or 
$(\gamma^{-1}\gamma^*)^2 = \gamma^{-1}\gamma^* - \id$.
\end{lm}

\begin{proof}
Write $x = \gamma^{-1}\gamma^*$.
Since $\gamma^*\gamma^{-1}\gamma^*$ is skew-symmetric,
it follows that $x^3 = -\id$.
Write $A' = \Ker (x + \id), A'' = \im ( x + \id)$.

First, we claim that $\Ker (x + \id) = \Ker (x + \id)^2$. To see this,
observe that if $(x + \id)^2 g = 0$, then since $x^3 = -\id$,
$0 = (x - 2\id)(x + \id)^2 g = (x^3 - 3x - 2\id) g = -3(x+\id) g$.
Since $q  \neq 3$, multiplication by $-3$ is invertible on $A$,
and therefore, $(x+\id) g = 0$.

It follows that $A = A' \oplus A''$.
Obviously, $x$ restricts to each component.
It is clear that on $A'$, $x$ is $-\id$. Since on $A''$,
$x + \id$ is invertible, and $0 = x^3 + \id = (x+\id)(x^2 - x + \id)$,
we have that $x^2 = x - \id$ on $A''$.
Now, we claim that $\gamma$ restricts to each component.
Since $\gamma: A \to A^*$, restricting to $\im (x + \id)$
means that $\gamma: \im (x + \id) \to \im (x + \id)^*$, and restricting
to $\Ker (x + \id)$ means that $\gamma: \Ker (x + \id) \to \Ker(x + \id)^*$.
The first follows from the fact that $\gamma(x+\id) = (x+\id)^*\gamma^*$,
and the second follows from $(\gamma - \gamma^*)(x + \id) = (x+\id)^*\gamma$.

We have shown that $A = A' \oplus A''$
where $x$ and $\gamma$ restrict to $A'$ and $A''$, $x$ is $-\id$ on $A'$,
and $x^2 = x - \id$ on $A''$. Denote by $n$ the unique natural number
so that $q^n A'' = 0$, but $q^{n-1} A'' \neq 0$.
 
We will show that $A''$ decomposes as a direct sum $\bigoplus_i{(C_i \oplus C_i)}$
which respects $\gamma$,
by strong induction on $|A''|$.
The base case, where $|A''| = 1$, is trivial.

In order to do the inductive step, first
suppose that $q^{n-1}\gamma(g,g) = 0$ for any $g$ such that there does not exist 
a $g'$ where $g = q g'$. Since any element of $A''$ is a multiple of some such $g$,
we would have that $q^{n-1}\gamma(g,g) = 0$ for any $g \in A''$.
From this, we would have that $q^{n-1} \gamma^* = -q^{n-1} \gamma$;
therefore, by the definition of $x$, we would have that
$\gamma^* = \gamma x$, we would have
$q^{n-1}x = -q^{n-1}\id$. Because $x^2 = x - \id$,
$q^{n-1}\id = -(q^{n-1}x) = x(-q^{n-1}\id) = q^{n-1}x^2 = q^{n-1}x - q^{n-1}\id
= -2q^{n-1}\id$. Therefore, $3q^{n-1}\id = 0$.
Since $q  \neq 3$, we would have $q^{n-1}\id = 0$, a contradiction.

Thus, we have that there is some $g$ such that there does not exist
a $g'$ with $g = q g'$ and such that $q^{n-1}\gamma(g,g) \neq 0$.
Write $B = \Ker \gamma g \cap \Ker \gamma^* g$. We claim that
$A'' = B \oplus \langle g \rangle \oplus \langle xg \rangle$.
To verify this, it suffices to show that the map
$(a,b) \to (\gamma(g,ag + bxg), \gamma^*(g,ag + bxg))$ which maps
$\ZZ/q^n\ZZ \times \ZZ/q^n\ZZ \to (\cup_{g \in A^*}{\im g})^2$
is an isomorphism. This is clear from the explicit computation that
$(\gamma(g,ag + bxg), \gamma^*(g,ag + bxg)) = ((a + b)s, as)$,
where $s = \gamma(g,g)$ is a generator of $\cup_{g \in A^*}{\im g}$.
Since $\Ker \gamma xg = \Ker \gamma^* g$,
and $\Ker \gamma^* xg = \Ker (\gamma^* - \gamma) g$, it is clear
that $B$ is orthogonal to $\langle g \rangle \oplus \langle xg \rangle$.
Applying the inductive hypothesis to $B$ completes the proof.

The proof that $A'$ decomposes as a direct sum $\bigoplus_i{(C_i \oplus C_i)}$
which respects $\gamma$, under the assumption that $\gamma$ is skew-symmetric,
is the standard proof that every skew-form has a symplectic basis
over a vector space, where instead of splitting off $\langle v, v' \rangle$ so that
$\gamma(v, v') \neq 0$, we split off $\langle g, g' \rangle$ such that 
the order of the cyclic subgroup generated by $g$ is $q^n$, and
so that $q^{n-1}\gamma(g,g') \neq 0$.
\end{proof}

\section{Proof of Theorem \ref{thmR3A}}\label{pfthmR3A}


By Lemma \ref{VectA}, the categorifications are in one to one
correspondence with an element of $\CoH^3(\ZZ/3\ZZ, \CC^*) \simeq
\ZZ/3\ZZ$ together with a map $\gamma: A \to A^*$ satisfying
$(\gamma^{-1}\gamma^*)^3 = -\id$.  In order to classify such forms
up to equivalence, it suffices to classify such forms on the 
$q$-parts of $A$ for each prime $q$.  By our Lemmas \ref{uniqueskew},
\ref{uniquegamma}, and \ref{q-gp}, there are $\prod{(a_i/2 + 1)}$
choices for $\gamma$, as on each $C_i \oplus C_i$, there are two
choices, depending on whether or not $\gamma$ is skew-symmetric.
Since there are three choices for the element of $\CoH^3(\ZZ/3\ZZ,
\CC^*) \simeq \ZZ/3\ZZ$, the statement of this corollary follows,
provided that we can show that $\Aut(\ZZ/3\ZZ)$ acts trivially on
$\CoH^3(\ZZ/3\ZZ, \CC^*)$ and our solutions for $\gamma$.  To see
that it acts trivially on the cohomology group, recall that
$\CoH^3(\ZZ/3\ZZ, \CC^*) = (\ZZ/3\ZZ)^{\otimes(-2)}$.  
To see that it acts trivially on our solutions, 
observe that $\gamma$ being skew-symmetric is the same
as $\gamma^*$ being skew-symmetric.  As $\gamma$ is determined by on how
many of each type of $C_i \oplus C_i$ it is skew-symmetric, $\gamma$ and
$\gamma^*$ are equivalent.  Thus, the statement of the theorem follows.



\begin{thebibliography}{99}



\bibitem[DGNO]{DGNO} V.~Drinfeld, S.~Gelaki, D.~Nikshych, and
V.~Ostrik. \emph{Group-theoretical properties of nilpotent modular
categories}. \url{http://arxiv.org/abs/0704.0195v2} [math.QA] 2 Apr 2007.


\bibitem[EGO]{pq} P.~Etingof, S.~Gelaki, and
V.~Ostrik. \emph{Classification of Fusion Categories of Dimension
$pq$}.  Int.\ Math.\ Res.\ Not.\ (2004), no.\ 57, 3041--3056.

\bibitem[ENO1]{ofc} P.~Etingof, D.~Nikshych, and V.~Ostrik. \emph{On
Fusion Categories}.  Ann.\ of Math.\ (2) 162 (2005), no.\ 2, 581--642.

\bibitem[ENO2]{ENO} P.~Etingof, D.~Nikshych, and V.~Ostrik. \emph{Weakly
group-theoretical and solvable fusion categories.}
\url{http://arxiv.org/abs/0809.3031v1} [math.QA] 17 Sep 2008.

\bibitem[ENO3]{ENOnew} P.~Etingof, D.~Nikshych, and
V.~Ostrik. \emph{Fusion categories and homotopy theory}. Preprint in
preparation.






\bibitem[GNN]{gnn} S.~Gelaki, D.~Naidu, and D.~Nikshych. \emph{Centers
of nilpotent fusion categories}. Preprint in preparation.

\bibitem[NZ]{NZ} W.~Nichols, M.~Zoeller, \emph{A Hopf algebra freeness
theorem.} Amer. J. Math.  111  (1989),  no. 2, 381--385.










\bibitem[O1]{Os1} V.~Ostrik. \emph{Module categories, weak Hopf algebras,
and modular invariants}.  Transformation Groups (2003) no. \ 8 177--206.

\bibitem[O2]{mcdd} V.~Ostrik. \emph{Module categories over the Drinfeld
double of a finite group.} Int.\ Math.\ Res.\ Not.\ (2003), no.\ 27,
1507--1520.

\bibitem[O3]{1nx} V.~Ostrik. \emph{Fusion categories of rank $2$}.
Math.\ Res.\ Lett.\ 10 (2003), no.\ 2-3, 177--183.




\bibitem[TY]{ty} D.~Tambara and S.~Yamagami. \emph{Tensor Categories with
fusion rules of self-duality for finite abelian groups}. J.\ Algebra 209
(1998), no.\ 2, 692--707.



\end{thebibliography}
\end{document}